\documentclass[oneside,article,]{memoir}
\usepackage{amsmath}         
\usepackage{amsthm,thmtools,thm-restate} 
\usepackage{enumitem}        

\usepackage{tikz}
  \usetikzlibrary{matrix,arrows,positioning}
  
\usepackage{tikz-cd}

\usepackage{float} 
\usepackage{caption, subcaption} 

\usepackage[nobysame,alphabetic,initials]{amsrefs} 
\usepackage[hidelinks]{hyperref}                   
\usepackage[capitalize,nosort]{cleveref}           

\usepackage{amssymb}
\usepackage[mathscr]{euscript}
\usepackage{relsize}
\usepackage{stmaryrd}
\usepackage{stackengine}

\usepackage{indentfirst} 
\usepackage{multicol}

\usepackage[inline,nomargin]{fixme} 
  \fxsetup{targetlayout=color}

\usepackage{float} 

\isopage[12]

\setlrmargins{*}{*}{1}
\checkandfixthelayout

\counterwithout{section}{chapter}
\usepackage{appendix}

  \newcommand{\id}[1][]{\operatorname{id}_{#1}}
  
  \newcommand{\cat}[1]{\mathscr{#1}}

  \declaretheorem[style=definition,within=section]{definition}
  \declaretheorem[style=definition,numberlike=definition]{example}

  \declaretheorem[style=plain,numberlike=definition]{corollary}
  \declaretheorem[style=plain,numberlike=definition]{lemma}
  \declaretheorem[style=plain,numberlike=definition]{proposition}
  \declaretheorem[style=plain,numberlike=definition]{theorem}
  \declaretheorem[style=plain,numberlike=definition]{conjecture}

  \declaretheorem[style=plain,numbered=no,name=Theorem]{theorem*}

  \Crefname{corollary}{Corollary}{Corollaries}
  \Crefname{definition}{Definition}{Definitions}
  \Crefname{lemma}{Lemma}{Lemmas}
  \Crefname{proposition}{Proposition}{Propositions}
  \Crefname{remark}{Remark}{Remarks}
  \Crefname{theorem}{Theorem}{Theorems}

  \newlist{axioms}{enumerate}{1}
  \Crefname{axiomsi}{}{}

  \newenvironment{tikzeq*}
  {
    \begingroup
    \begin{equation*}
    \begin{tikzpicture}[baseline=(current bounding box.center)]
  }
  {
    \end{tikzpicture}
    \end{equation*}
    \endgroup
    \ignorespacesafterend
  }

  \tikzset
  {
    diagram/.style=
    {
      matrix of math nodes,
      column sep={4.3em,between origins},
      row sep={4em,between origins},
      text height=1.5ex,
      text depth=.25ex
    },
    over/.style={preaction={draw=white,-,line width=6pt}},
    every to/.style={font=\footnotesize},
    inj/.style={right hook->},
    surj/.style={-{Latex[open]}},
    cof/.style={>->},
    fib/.style={->>},
  }

  \DeclareFontFamily{U}{mathx}{\hyphenchar\font45}

  \DeclareFontShape{U}{mathx}{m}{n}{
    <5> <6> <7> <8> <9> <10>
    <10.95> <12> <14.4> <17.28> <20.74> <24.88>
    mathx10}{}

  \DeclareSymbolFont{mathx}{U}{mathx}{m}{n}

  \DeclareFontFamily{U}{mathb}{\hyphenchar\font45}

  \DeclareFontShape{U}{mathb}{m}{n}{
    <5> <6> <7> <8> <9> <10>
    <10.95> <12> <14.4> <17.28> <20.74> <24.88>
    mathb10}{}

  \DeclareSymbolFont{mathb}{U}{mathb}{m}{n}

  \DeclareMathAccent{\widebar}{0}{mathx}{"73}

  \DeclareMathSymbol{\Rsh}{\mathrel}{mathb}{"E9}

  \DeclareFontFamily{U}{MnSymbolA}{}

  \DeclareFontShape{U}{MnSymbolA}{m}{n}{
    <-6> MnSymbolA5
    <6-7> MnSymbolA6
    <7-8> MnSymbolA7
    <8-9> MnSymbolA8
    <9-10> MnSymbolA9
    <10-12> MnSymbolA10
    <12-> MnSymbolA12}{}

  \DeclareSymbolFont{MnSyA}{U}{MnSymbolA}{m}{n}

  \DeclareMathSymbol{\twoheaddownarrow}{\mathrel}{MnSyA}{27}

  \newcommand{\MSC}[1]{%
    \let\thempfn\relax
    \footnotetext[0]{2020 Mathematics Subject Classification: #1.}
  }

\tikzstyle{vertex}=[circle, draw, minimum size=7pt, inner sep=0pt]

\newcommand{\Cat}{\mathsf{Cat}} 
\newcommand{\Graph}{\mathsf{Graph}} 
\newcommand{\DiGraph}{\mathsf{DiGraph}} 
\newcommand{\Set}{\mathsf{Set}} 

\numberwithin{equation}{section}

\author{Adrien Grenier \and Krzysztof Kapulkin} 
\title{Biclosed monoidal structures on the categories of digraphs and graphs}
\date{\today}

\begin{document}

  \maketitle

  \begin{abstract}
    We show that the categories of directed and undirected reflexive graphs carry exactly two (up to isomorphism) biclosed monoidal structures.
    \MSC{05C76, 18D15 (primary), 05C25 (secondary)}
  \end{abstract}

\section*{Introduction}

The study of different homotopy theories defined within the category of graphs, often termed \emph{discrete homotopy theory}, has seen renewed interest in recent years \cite{barcelo-greene-jarrah-welker:comparison,barcelo-greene-jarrah-welker:vanishing,chih-scull:groupoid,dochtermann-singh,carranza-doherty-kapulkin,carranza-kapulkin:cubical-graphs} thanks to their applications to a wide range of settings, including: topological data analysis \cite{memoli-zhou,kapulkin-kershaw:data-analysis}, graph coloring \cite{dochtermann:hom-complex,babson-kozlov:lovasz}, and network analysis \cite{chowdhury-memoli:persistent-path-homology}.
Amongst the prominent notions, there are: the A-theory \cite{babson-barcelo-longueville-laubenbacher,barcelo-kramer-laubenbacher-weaver,barcelo-laubenbacher}, the $\times$-theory \cite{dochtermann:hom-complex}, and the GLMY-theory \cite{grigoryans}.

The work of Rieser on (\v{C}ech) closure spaces \cite{rieser:closure-spaces}, and subsequent contributions by Bubenik and Milicevic \cite{bubenik-milicevic} show the wealth of (naive) homotopy theories that the category of graphs, viewed as a full subcategory of the category of closure spaces, can support; such a theory arises for every choice of an interval and a graph product.
(The word `naive' here is meant to distinguish the homotopy theories generally considered in combinatorics from the more general ones often considered in abstract homotopy theory \cite{dwyer-hirschhorn-kan-smith}.)
To know that we are not missing any potentially useful homotopy theories, we would like to determine possible choices of the interval and the graph product.
This leads to the motivating question of the present paper: How many choices of the product are there on the category of graphs?

To make this question meaningful, we need to put some restrictions on the notion of product; a natural way of phrasing these conditions would be to ask that this product makes the category of graphs into a biclosed monoidal category.
Loosely speaking, this means that the graph product under consideration is associative (up to graph isomorphism), unital, and that it induces the structure of a graph on the set of graph homomorphisms between two fixed graphs.
With this notion in place, we can now restate our motivating question to ask: how many biclosed monoidal structures are there on the category of graphs?

In the present paper, we show that the categories of directed and undirected reflexive graphs carry exactly two (up to isomorphism) biclosed monoidal structures each, given by the categorical product and the box product.
Kershaw and the second-named author \cite{kapulkin-kershaw:monoidal-graphs} recently gave a proof of this fact for undirected reflexive graphs, using a combinatorially involved argument.
In comparison, our proof is conceptual, inspired by the theory of finite limit sketches \cite{folz-lair-kelly,bourke-gurski:cocategorical-obstruction}.
The theory itself will not be used, as the proof requires only rudimentary category theory, but it was essential in our discovering the proof in the first place.
This leaves the question of determining all possible biclosed monoidal structures on the category of (directed or undirected) graphs that are not necessarily reflexive.
We discuss the limitations of the present methods to tackle this question in the final section of the paper.

\textbf{Acknowledgements.}
We thank Ross Street for suggesting the theory of finite-limit sketches to us and making us aware of the paper \cite{folz-lair-kelly}.
The authors were supported by the Natural Sciences and Engineering Research Council (NSERC) of Canada via the Undergraduate Student Research Award (A.G.) and the Discovery Grant (K.K.).
We thank NSERC for its generosity.

\section{Preliminaries}

In this section, we review the necessary categorical language to phrase and prove our results.
We begin by recasting the theory of directed reflexive graphs, as used in the GLMY-theory, in the categorical language.

\begin{definition}
    A \textit{directed graph} is a set  with a reflexive relation, where $x \rightsquigarrow y$ means there is an edge from $x$ to $y$. A \textit{directed graph homomorphism} is a function between the sets of vertices which preserves the relation. Together, these form the category $\DiGraph$ of directed graphs.
\end{definition}

We analyze the category of graphs by expressing it as a reflective subcategory of the category of functors $\mathbb{D} \rightarrow \Set$ and natural transformations. 

\begin{definition}
    Let $\mathbb{D}$ be the category generated by the diagram

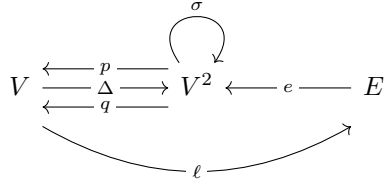
\[\begin{tikzcd}
	V && {V^2} && E
	\arrow["\Delta"{description}, from=1-1, to=1-3]
	\arrow["\ell"{description}, shift right=4, bend right, from=1-1, to=1-5]
	\arrow["q"{description}, shift left=3, from=1-3, to=1-1]
	\arrow["p"{description}, shift right=3, from=1-3, to=1-1]
	\arrow["{\sigma}", from=1-3, to=1-3, loop, in=55, out=125, distance=10mm]
	\arrow["e"{description}, from=1-5, to=1-3]
\end{tikzcd}\]
    subject to the following equations:
    \begin{align*}
        p\Delta &= q\Delta = \text{id}_V    &  p\sigma &= q     &   \sigma \Delta &= \Delta \\
        \Delta &= e\ell                     &  q\sigma &= p     &   \sigma^2 &= \text{id}_{V^2} \\
    \end{align*}
\end{definition}

Explicitly, the morphisms between any pair of objects are presented in the following table:
\begin{center}
\renewcommand{\arraystretch}{1.5}
    \begin{tabular}{cc|ccc} 
         & & & codomain \\
         & & $V$ & $V^2$ & $E$ \\
         \hline
         & $V$ & $\{\text{id}_V\}$ & $\{\Delta\}$ & $\{\ell\}$ \\
         domain & $V^2$ & $\{p, q\}$ & $\{\text{id}_{V^2}, \Delta p, \Delta q, \sigma\}$ & $\{\ell p, \ell q\}$ \\
         & $E$ & $\{pe, qe\}$ & $\{e, \Delta pe, \Delta qe, \sigma e\}$ & $\{\ell pe, \ell qe, \text{id}_E\}$ \\
    \end{tabular}
\end{center}

By consulting this table, we readily see that:

\begin{proposition}
    In $\mathbb{D}$:
    \begin{enumerate}
        \item $e$ is monic; and
        \item $V^2$ is the product $V \times V$ with projections $p,q$. \qed
    \end{enumerate}
\end{proposition}

\begin{corollary}\label{cor:digraph-as-fin-limit-pres-functors}
    $\DiGraph$ is equivalent to the full subcategory of finite limit preserving functors from $\mathbb{D}$ to $\Set$. \qed
\end{corollary}

In light of \cref{cor:digraph-as-fin-limit-pres-functors}, we view the presheaf category $\Set^\mathbb{D}$ of all functors as the  category of marked directed multigraphs.
Explicitly, an object $X \in \Set^\mathbb{D}$ consists of a set $X_V$ of vertices, a multiset $X_{V^2}$ of ordered pairs of vertices and a multiset $X_E$ of directed edges. 
Since finite limit preserving functors form a right class in a factorization system on $\Set^\mathbb{D}$, we conclude that $\DiGraph$ is a reflective subcategory of $\Set ^\mathbb{D}$.

The theory of undirected graphs can be developed similarly.

\begin{definition}
    An \textit{undirected graph} is a set with a reflexive, symmetric relation. An \textit{undirected graph homomorphism} is a function between the sets of vertices which preserves the relation. Together, these form the category $\Graph$ of undirected graphs
\end{definition}

\begin{definition}
    Let $\mathbb{G}$ be the category generated by the diagram
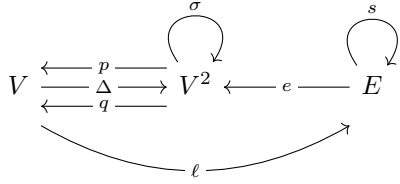
\[\begin{tikzcd}
	V && {V^2} && E
	\arrow["\Delta"{description}, from=1-1, to=1-3]
	\arrow["\ell"{description}, shift right=4, bend right, from=1-1, to=1-5]
	\arrow["q"{description}, shift left=3, from=1-3, to=1-1]
	\arrow["p"{description}, shift right=3, from=1-3, to=1-1]
	\arrow["{\sigma}", from=1-3, to=1-3, loop, in=55, out=125, distance=10mm]
	\arrow["e"{description}, from=1-5, to=1-3]
	\arrow["s", from=1-5, to=1-5, loop, in=55, out=125, distance=10mm]
\end{tikzcd}\]    
subject to the following equations:
    \begin{align*}
        p\Delta &= q\Delta = \text{id}_V    & p\sigma &= q  & \sigma \Delta &= \Delta       & es &= \sigma e \\
        \Delta &= e\ell                     & q\sigma &= p  & \sigma^2 &= \text{id}_{V^2}   & s\ell &= \ell \\
                &                           &         &     &           &                   & s^2 &= \text{id}_E
    \end{align*}
\end{definition}

All previous results hold for the case of undirected graphs and $\mathbb{G}$.
We record them, but omit their proofs for brevity.

\begin{proposition} \leavevmode
    \begin{enumerate}
        \item In $\mathbb{G}$: $e$ is monic and $V^2$ is the product $V \times V$ with projections $p$ and $q$.
        \item $\Graph$ is equivalent to the full subcategory of finite limit preserving functors from $\mathbb{G}$ to $\Set$.
        \item $\Graph$ is a reflective subcategory of $\Set ^{\mathbb{G}}$. \qed
    \end{enumerate}
\end{proposition}

We write $J_n$ for the directed graph with $n+1$ vertices where the $i$th vertex has an edge to the $(i +1)$st vertex for $i<n$. 
Similarly, we write $I_n$ for the undirected graph with $n+1$ vertices connected in the same way but with undirected edges. 

We now quickly review biclosed monoidal categories.

\begin{definition}
    A \textit{monoidal category} consists of a category $\cat{C}$, along with a bifunctor $\otimes \colon \cat{C} \times \cat{C} \rightarrow \cat{C}$ (the tensor product), an object $I \in \cat{C}$ (the unit) and three natural isomorphisms $\alpha, \lambda, \varrho$ with components:
    \[\alpha_{X, Y, Z} \colon (X \otimes Y) \otimes Z \cong X \otimes (Y \otimes z), \qquad \lambda_X \colon I \otimes X \cong X, \text{ and} \qquad \varrho_x \colon X \otimes I \cong X; \]
    making the following diagrams commute for all $X, Y, Z, W \in \cat{C}$: 
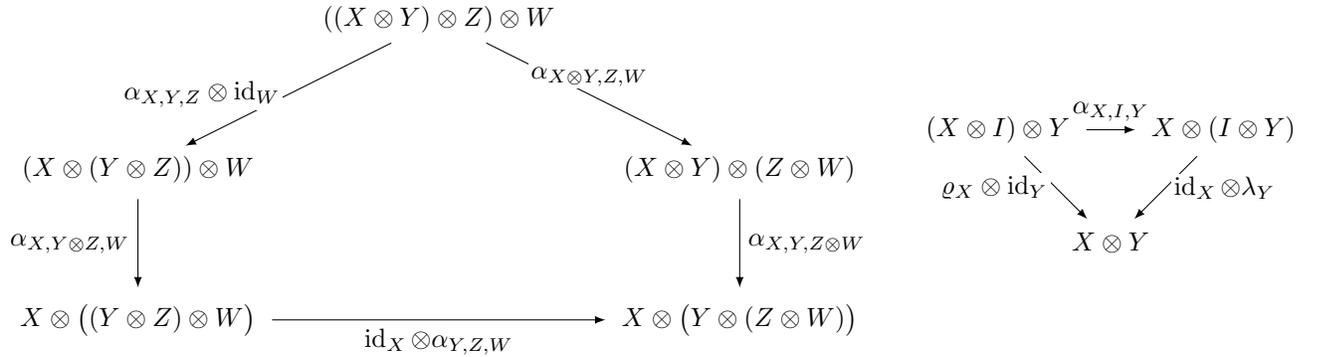
\[
\noindent
\begin{tikzpicture}[
    baseline=(current bounding box.center),
    scale=2,
    every node/.style={inner sep=2pt},
    arrow/.style={->,>=latex,shorten >=4pt,shorten <=4pt},
    lab/.style={midway, fill=white, inner sep=1pt}
  ]
  \coordinate (A) at (0,2);    
  \coordinate (B) at (2,1);    
  \coordinate (C) at (2,0);    
  \coordinate (D) at (-2,0);   
  \coordinate (E) at (-2,1);   

  \node (NA) at (A) {$((X\otimes Y)\otimes Z)\otimes W$};
  \node (NB) at (B) {$(X\otimes Y)\otimes (Z\otimes W)$};
  \node (NC) at (C) {$X\otimes\bigl(Y\otimes (Z\otimes W)\bigr)$};
  \node (ND) at (D) {$X\otimes\bigl((Y\otimes Z)\otimes W\bigr)$};
  \node (NE) at (E) {$(X\otimes (Y\otimes Z))\otimes W$};

  \draw[arrow] (NA) -- node[lab, above=2pt] {$\alpha_{X\otimes Y,Z,W}$} (NB);
  \draw[arrow] (NB) -- node[lab, right=2pt] {$\alpha_{X,Y,Z\otimes W}$} (NC);
  \draw[arrow] (NA) -- node[lab, left=2pt] {$\alpha_{X,Y,Z}\otimes \id[W]$} (NE);
  \draw[arrow] (NE) -- node[lab, left=2pt] {$\alpha_{X,Y\otimes Z,W}$} (ND);
  \draw[arrow] (ND) -- node[lab, below=2pt] {$\id[X]\otimes \alpha_{Y,Z,W}$} (NC);
\end{tikzpicture}
\qquad
\begin{tikzpicture}[
    baseline=(current bounding box.center),
    every node/.style={inner sep=2pt},
    arrow/.style={->,>=latex,shorten >=4pt,shorten <=4pt},
    lab/.style={midway, fill=white, inner sep=1pt}
  ]
  \coordinate (A) at (0,1.5);   
  \coordinate (B) at (3,1.5);   
  \coordinate (C) at (1.5,0);   

  \node (NA) at (A) {$(X\otimes I)\otimes Y$};
  \node (NB) at (B) {$X\otimes (I\otimes Y)$};
  \node (NC) at (C) {$X\otimes Y$};

  \draw[arrow] (NA) -- node[lab, above=2pt] {$\alpha_{X,I,Y}$} (NB);
  \draw[arrow] (NA) -- node[lab, left=2pt] {$\varrho_X\otimes \id[Y]$} (NC);
  \draw[arrow] (NB) -- node[lab, right=2pt] {$\id[X]\otimes \lambda_Y$} (NC);
\end{tikzpicture}
\]
\end{definition}

\begin{definition}
A monoidal category $(\cat{C}, \otimes, I)$ is \emph{left closed} (respectively \emph{right closed}) if for every object $X \in \cat{C}$, the functor $ X \otimes - $ has a right adjoint (respectively $ - \otimes X$ has a right adjoint).
A \emph{biclosed} monoidal category is a monoidal category that is both left closed and right closed.
\end{definition}

Examples of biclosed monoidal categories include $\Set$ equipped with the cartesian product, $\mathsf{Ab}$ equipped with the tensor product, and $\Cat$ equipped with the cartesian product.
There are also two biclosed monoidal structures on the category of (di)graphs which we now define.

\begin{definition}
    Let $G, H$ be graphs (either both directed or undirected).
    \begin{itemize}
        \item The \emph{box product} $G \mathbin{\square} H$ has vertex set $G(V) \times H(V)$ and $(g,h) \sim (g',h')$ (there is an edge from $(g,h)$ to $(g',h')$) if $g = g'$ and $h \sim h'$, or $g \sim g'$ and $h = h'$.
        \item the \emph{categorical product} $G \boxtimes H$ has vertex set $G(V) \times H(V)$ and $(g,h) \sim (g',h')$ if $g = g'$ and $h \sim h'$, or $g \sim g'$ and $h = h'$, or $g \sim g'$ and $h \sim h'$.
    \end{itemize}
\end{definition}

It is easy to check that both the box product and the categorical product define a biclosed monoidal structure on both $\DiGraph$ and $\Graph$.




\begin{lemma}[cf.~{\cite[Cor.~3.4]{kapulkin-kershaw:monoidal-graphs}}] \label{lem:unit} \leavevmode
    \begin{enumerate}
        \item A biclosed monoidal structure on $\DiGraph$ must have unit $J_0$.
        \item A biclosed monoidal structure on $\Graph$ must have unit $I_0$.
    \end{enumerate}
\end{lemma}

\begin{proof}
We only prove (1), as the proof of (2) is identical.
For (1), suppose that $(\DiGraph, \otimes, J)$ is a biclosed monoidal category.
Since $\DiGraph$ is well-pointed, by \cite[Prop.~1.7]{kapulkin-kershaw:monoidal-graphs}, the unit is a subobject of the terminal object $J_0$; that is, $J$ is either $\varnothing$ or $J_0$.
The monoidal structure is biclosed, so $\otimes$ preserves the initial object in each variable. 
If $J = \varnothing$, then for every digraph $G$, we have $G = G \otimes \varnothing = \varnothing$.
Hence, $J$ must be $J_0$.
\end{proof}


\begin{lemma} \label{lem:determined-by-square} \leavevmode
    \begin{enumerate}
        \item Any biclosed monoidal structure $(\DiGraph, \otimes, J_0)$ is completely determined by $J_1 \otimes J_1$.
        \item Any biclosed monoidal structure $(\Graph, \otimes, I_0)$ is completely determined by $I_1 \otimes I_1$.
    \end{enumerate}
\end{lemma}

\begin{proof}
By the density theorem, the monoidal product is determined by its action on $J_0$ and $J_1$ (respectively $I_0$ and $I_1$).
Since $J_0$ (respectively $I_0$) is the monoidal unit, the result follows.
\end{proof}

\section{Directed graphs}

In this section, we prove that the category of directed reflexive graphs carries exactly two biclosed monoidal structures.
We begin with a piece of notation and a crucial technical lemma.

Let $C_{Sq}$ be the 'commutative square' graph:
\[\begin{tikzcd}
	\bullet && \bullet \\
	\\
	\bullet && \bullet
	\arrow[from=1-1, to=1-3]
	\arrow[from=1-1, to=3-1]
	\arrow[from=1-3, to=3-3]
	\arrow[from=3-1, to=3-3]
\end{tikzcd}\]

\begin{lemma} \label{lem:J1-J1-four-vertices}
    If $(\DiGraph, \otimes, J_0)$ is a biclosed monoidal structure, then $J_1 \otimes J_1$ must have exactly four vertices and contains $C_{Sq}$ as a subgraph.
\end{lemma}

\begin{proof}
By applying the Yoneda embedding to $\mathbb{D}^{op}$ we get a codigraph in $\DiGraph$: 
\[\begin{tikzcd}
	{J_0} && {J_0 \sqcup J_0} && {J_1}
	\arrow["p"{description}, shift left=3, from=1-1, to=1-3]
	\arrow["q"{description}, shift right=3, from=1-1, to=1-3]
	\arrow["i"{description}, from=1-3, to=1-1]
	\arrow["e", two heads, from=1-3, to=1-5]
	\arrow["\ell", shift left=4, bend left, two heads, from=1-5, to=1-1]
\end{tikzcd}\]
It follows that $p$ and $q$ send the unique vertex of $J_0$ to distinct vertices in $J_0 \sqcup J_0$ and that $e$ acts as identity on vertices.

By assumption, the functor $\otimes$ preserves colimits, and thus codigraph objects.
Taking the product $\otimes$ of the above digraph with itself gives the following $3 \times 3$ serially commuting square of digraphs where each row and column is a codigraph:
\[\begin{tikzcd}
	& {- \otimes J_0:} && {- \otimes (J_0 \sqcup J_0):} && {- \otimes J_1:} \\
	{J_0 \otimes -:} & {J_0 \otimes J_0} && {J_0 \otimes (J_0 \sqcup J_0)} && {J_0 \otimes J_1} \\
	\\
	{(J_0 \sqcup J_0 )\otimes -:} & {(J_0 \sqcup J_0 ) \otimes (J_0)} && {(J_0 \sqcup J_0 ) \otimes (J_0 \sqcup J_0)} && {(J_0 \sqcup J_0 ) \otimes J_1} \\
	\\
	{J_1 \otimes -:} & {J_1 \otimes J_0} && {J_1 \otimes (J_0 \sqcup J_0 )} && {J_1 \otimes J_1}
	\arrow[shift left=3, from=2-2, to=2-4]
	\arrow[shift right=3, from=2-2, to=2-4]
	\arrow[shift left=3, from=2-2, to=4-2]
	\arrow[shift right=3, from=2-2, to=4-2]
	\arrow[from=2-4, to=2-2]
	\arrow[two heads, from=2-4, to=2-6]
	\arrow[shift left=3, from=2-4, to=4-4]
	\arrow[shift right=3, from=2-4, to=4-4]
	\arrow[shift left=3, from=2-6, to=4-6]
	\arrow[shift right=3, from=2-6, to=4-6]
	\arrow[from=4-2, to=2-2]
	\arrow[shift left=3, from=4-2, to=4-4]
	\arrow[shift right=3, from=4-2, to=4-4]
	\arrow[two heads, from=4-2, to=6-2]
	\arrow[from=4-4, to=2-4]
	\arrow[from=4-4, to=4-2]
	\arrow[two heads, from=4-4, to=4-6]
	\arrow[two heads, from=4-4, to=6-4]
	\arrow[from=4-6, to=2-6]
	\arrow[two heads, from=4-6, to=6-6]
	\arrow[shift left=3, from=6-2, to=6-4]
	\arrow[shift right=3, from=6-2, to=6-4]
	\arrow[from=6-4, to=6-2]
	\arrow[two heads, from=6-4, to=6-6]
\end{tikzcd}\]
which simplifies to:
\begin{equation}
\begin{tikzcd}
	{J_0} && {J_0 \sqcup J_0} && {J_1} \\
	\\
	{J_0 \sqcup J_0} && {J_0 \sqcup J_0 \sqcup J_0 \sqcup J_0} && {J_1 \sqcup J_1} \\
	\\
	{J_1 } && {J_1 \sqcup J_1} && {J_1 \otimes J_1}
	\arrow[shift left=3, from=1-1, to=1-3]
	\arrow[shift right=3, from=1-1, to=1-3]
	\arrow[shift left=3, from=1-1, to=3-1]
	\arrow[shift right=3, from=1-1, to=3-1]
	\arrow[from=1-3, to=1-1]
	\arrow[two heads, from=1-3, to=1-5]
	\arrow[shift left=3, from=1-3, to=3-3]
	\arrow[shift right=3, from=1-3, to=3-3]
	\arrow[shift left=3, from=1-5, to=3-5]
	\arrow[shift right=3, from=1-5, to=3-5]
	\arrow[from=3-1, to=1-1]
	\arrow[shift left=3, from=3-1, to=3-3]
	\arrow[shift right=3, from=3-1, to=3-3]
	\arrow[two heads, from=3-1, to=5-1]
	\arrow[from=3-3, to=1-3]
	\arrow[from=3-3, to=3-1]
	\arrow[two heads, from=3-3, to=3-5]
	\arrow[two heads, from=3-3, to=5-3]
	\arrow[from=3-5, to=1-5]
	\arrow[two heads, from=3-5, to=5-5]
	\arrow[shift left=3, from=5-1, to=5-3]
	\arrow[shift right=3, from=5-1, to=5-3]
	\arrow[from=5-3, to=5-1]
	\arrow[two heads, from=5-3, to=5-5]
\end{tikzcd}
    \label{diagram}
\end{equation}
(Note that we have omitted the epimorphisms into the top row and left column for clarity).
We now determine $J_1 \otimes J_1$.
Labeling the vertices $a,b,c,d$, the action of the maps on vertices is as follows:
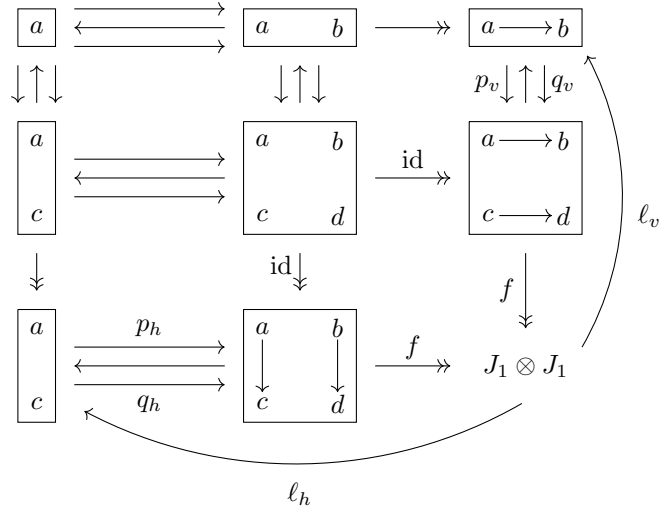
\tikzset{
 ab/.pic = {
 \node at (0,0) {$a$} ;
 \node at (1,0) {$b$} ;
 \draw[->] (0.15,0) -- (0.85,0) ;
 }
}

\tikzset{
 ab_boxed/.pic = {
 \pic at (0,0) {ab} ;
 \draw (-.25, -.25) -- (1.25, -.25) -- (1.25, .25) -- (-.25, .25) -- (-.25, -.25);
 }
}

\tikzset{
 ab_none/.pic = {
 \node at (0,0) {$a$} ;
 \node at (1,0) {$b$} ;
 }
}

\tikzset{
 ab_none_boxed/.pic = {
 \pic at (0,0) {ab_none} ;
 \draw (-.25, -.25) -- (1.25, -.25) -- (1.25, .25) -- (-.25, .25) -- (-.25, -.25);
 }
}

\tikzset{
 ac/.pic = {
 \node at (0,0) {$a$} ;
 \node at (0,-1) {$c$} ;
 \draw[->] (0,-.15) -- (0,-.85) ;
 }
}

\tikzset{
 ac_boxed/.pic = {
 \pic at (0,0) {ac} ;
 \draw (-.25, .25) -- (.25, .25) -- (.25, -1.25) -- (-.25, -1.25) -- (-.25, .25);
 }
}

\tikzset{
 ac_none/.pic = {
 \node at (0,0) {$a$} ;
 \node at (0,-1) {$c$} ;
 }
}

\tikzset{
 ac_none_boxed/.pic = {
 \pic at (0,0) {ac_none} ;
 \draw (-.25, .25) -- (.25, .25) -- (.25, -1.25) -- (-.25, -1.25) -- (-.25, .25);
 }
}

\tikzset{
 a_none/.pic = {
 \node at (0,0) {$a$};
 }
}

\tikzset{
 a_none_boxed/.pic = {
 \pic at (0,0) {a_none} ;
 \draw (-0.25,-0.25) -- (-0.25,0.25) -- (0.25,0.25) -- (0.25,-0.25) -- (-0.25,-0.25);
 }
}

\tikzset{
 abcd_none/.pic = {
 \node at (0,0) {$a$};
 \node at (1,0) {$b$};
 \node at (0,-1) {$c$};
 \node at (1,-1) {$d$};
 }
}

\tikzset{
 abcd_none_boxed/.pic = {
 \pic at (0,0) {abcd_none};
 \draw (-.25,.25) -- (-.25, -1.25) -- (1.25, -1.25) -- (1.25, .25) -- (-.25,.25);
 }
}

\tikzset{
 abcd_ver/.pic = {
 \pic at (0,0) {abcd_none};
 \draw[->] (0, -.15) -- (0, -0.85) ;
 \draw[->] (1, -.15) -- (1, -0.85) ;
 }
}

\tikzset{
 abcd_ver_boxed/.pic = {
 \pic at (0,0) {abcd_ver} ;
 \draw (-.25,.25) -- (-.25, -1.25) -- (1.25, -1.25) -- (1.25, .25) -- (-.25,.25);
 }
}

\tikzset{
 abcd_hor/.pic = {
 \pic at (0,0) {abcd_none};
 \draw[->] (0.15, 0) -- (0.85, 0) ;
 \draw[->] (.15, -1) -- (.85, -1) ;
 }
}

\tikzset{
 abcd_hor_boxed/.pic = {
 \pic at (0,0) {abcd_hor} ;
 \draw (-.25,.25) -- (-.25, -1.25) -- (1.25, -1.25) -- (1.25, .25) -- (-.25,.25);
 }
}

\tikzset{
 abcd_full/.pic = {
 \pic at (0,0) {abcd_none};
\draw[->] (0.15, 0) -- (0.85, 0) ;
 \draw[->] (.15, -1) -- (.85, -1) ;
 \draw[->] (0, -.15) -- (0, -0.85) ;
 \draw[->] (1, -.15) -- (1, -0.85) ;
 }
}

\tikzset{
 abcd_full_boxed/.pic = {
 \pic at (0,0) {abcd_full} ;
 \draw (-.25,.25) -- (-.25, -1.25) -- (1.25, -1.25) -- (1.25, .25) -- (-.25,.25);
 }
}

\begin{center}
\begin{tikzpicture}

    \pic at (0,0) {a_none_boxed} ;
    \pic at (3,0) {ab_none_boxed} ;
    \pic at (6,0) {ab_boxed} ;

    \draw[->] (2.5,0) -- (.5,0);
    \draw[->] (.5,.25) -- (2.5,.25);
    \draw[->] (.5,-.25) -- (2.5,-.25);
    \draw[->>] (4.5,0) -- (5.5,0) ;

    \pic at (0,-1.5) {ac_none_boxed} ;
    \pic at (3,-1.5) {abcd_none_boxed} ;
    \pic at (6,-1.5) {abcd_hor_boxed} ;

    \draw[->] (2.5,-2) -- (.5,-2);
    \draw[->] (.5,-1.75) -- (2.5,-1.75);
    \draw[->] (.5,-2.25) -- (2.5,-2.25);
    \draw[->>] (4.5,-2) -- (5.5,-2) ;

    \pic at (0,-4) {ac_none_boxed} ;
    \pic at (3,-4) {abcd_ver_boxed} ;
    \node at (6.5,-4.5) {$J_1 \otimes J_1$} ;

    \draw[->] (2.5,-4.5) -- (.5,-4.5);
    \draw[->] (.5,-4.25) -- (2.5,-4.25);
    \draw[->] (.5,-4.75) -- (2.5,-4.75);
    \draw[->>] (4.5,-4.5) -- (5.5,-4.5) ;

    \draw[->] (0,-1) -- (0,-.5);
    \draw[->] (.25,-.5) -- (.25,-1);
    \draw[->] (-.25,-.5) -- (-0.25,-1);
    \draw[->>] (0,-3) -- (0,-3.5) ;

    \draw[->] (3.5,-1) -- (3.5,-.5);
    \draw[->] (3.75,-.5) -- (3.75,-1);
    \draw[->] (3.25,-.5) -- (3.25,-1);
    \draw[->>] (3.5,-3) -- (3.5,-3.5) ;

    \draw[->] (6.5,-1) -- (6.5,-.5);
    \draw[->] (6.75,-.5) -- (6.75,-1);
    \draw[->] (6.25,-.5) -- (6.25,-1);
    \draw[->>] (6.5,-3) -- (6.5,-4) ;

    \node at (1.5,-4) {$p_h$} ;
    \node at (1.5,-5) {$q_h$} ;
    \node at (5,-4.25) {$f$} ;
    
    \node at (6,-.75) {$p_v$} ;
    \node at (7,-.75) {$q_v$} ;
    \node at (6.25,-3.5) {$f$} ;

    \node at (3.25, -3.15) {id} ;
    \node at (5, -1.75) {id} ;

    \draw[->] (6.45,-5) arc (-60:-118:6);
    \draw[->] (7.25,-4.25) arc (330:388:4);

    \node at (3.5,-6.2) {$\ell _h$} ;
    \node at (8.15,-2.5) {$\ell _v$} ;
    
\end{tikzpicture}
\end{center}

where the maps $p_h,q_h, p_v, q_v$ are as follows:
\begin{center}
    \begin{tabular}{|c|c|}
    \hline
    $p_h: \begin{cases}
        a \mapsto a \\
        c \mapsto c
    \end{cases}$ 
    &
    $q_h: \begin{cases}
        a \mapsto b \\
        c \mapsto d
    \end{cases}$ \\
    \hline
    $p_v: \begin{cases}
        a \mapsto a \\
        b \mapsto b
    \end{cases}$ 
    &
    $q_v: \begin{cases}
        a \mapsto c \\
        b \mapsto d
    \end{cases}$ \\
    \hline
    \end{tabular}
\end{center}
We have that $\ell_\bullet fp_\bullet = \text{id} = \ell_\bullet f q_\bullet$.
Therefore:
\begin{align}
    \ell_h f(a) &= \ell_h f(b) = a \\
    \ell_h f(c) &= \ell_h f(d) = c \\ 
    \ell_v f(a) &= \ell_v f(c) = a \\
    \ell_v f(b) &= \ell_v f(d) = b
\end{align}
Because we we have an epimorphism $J_1 \sqcup J_1 \rightarrow J_1 \otimes J_1$, we know that $J_1 \otimes J_1$ has at most 4 vertices. The equations (2.2) to (2.5) show that the vertices $f(a), f(b), f(c), f(d)$ are distinct (for example, $f(a)$ is distinct from the other vertices since (2.2) and (2.3) show that $f(a) \neq f(c), f(a) \neq f(d)$ and (2.4) and (2.5) show that $f(a) \neq f(b)$). Then, $J_1 \otimes J_1$ has at least 4 vertices, so it has exactly 4 vertices.
\end{proof}

With the preliminaries now in place, we can prove the main theorem of this section.

\begin{theorem}
    The only biclosed monoidal structures on $\DiGraph$ are $\mathbin{\square}$ and $\boxtimes$.
\end{theorem}

\begin{proof}
Continuing with the notation of \cref{lem:J1-J1-four-vertices}, the maps $fp_h, fq_h, fp_v, fq_v \colon J_1 \rightarrow J_1 \otimes J_1$ all pick out a different edge, so that 
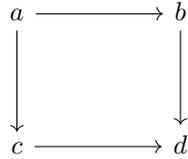
\[\begin{tikzcd}
	a && b \\
	\\
	c && d
	\arrow[from=1-1, to=1-3]
	\arrow[from=1-1, to=3-1]
	\arrow[from=1-3, to=3-3]
	\arrow[from=3-1, to=3-3]
\end{tikzcd}\]
is a subgraph of $J_1 \otimes J_1$. 

If there was another edge in the opposite direction of a known edge, say $b \rightsquigarrow a$ in $J_1 \otimes J_1$, then $\ell_v$ implies $b \rightsquigarrow a$ in $J_1$. However, there is no such edge so $b \not \rightsquigarrow a$ in $J_1 \otimes J_1$. Similar reasoning yields that there are no edges in opposite directions of the known edges, $d \not \rightsquigarrow a$, $b \not \rightsquigarrow c$ and $c \not \rightsquigarrow b$. The only case left to determine is whether there is an edge $a \rightsquigarrow d$. If there is no such edge, then the monoidal structure must be the box product. If there is an edge, then the monoidal structure must be the categorical product. Therefore, these are the only possible such structures.
\end{proof}

\section{Undirected graphs}

The proof for undirected reflexive graph is almost identical, so we only discuss the necessary changes that need to be made to the arguments of the previous section.
Let $C_4$ denote the undirected cyclic graph on 4 vertices:
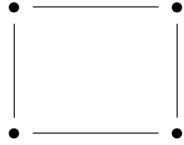
\[\begin{tikzcd}
	\bullet && \bullet \\
	\\
	\bullet && \bullet
	\arrow[no head, from=1-1, to=1-3]
	\arrow[no head, from=1-1, to=3-1]
	\arrow[no head, from=1-3, to=3-3]
	\arrow[no head, from=3-1, to=3-3]
\end{tikzcd}\]

\begin{theorem}
    The only biclosed monoidal structures on $\Graph$ are $\mathbin{\square}$ and $\boxtimes$.
\end{theorem}

\begin{proof}
We proceed as in $\S 3$. 
Suppose that $(\Graph, \otimes, I_0)$ is a biclosed monoidal structure. 
By applying the Yoneda embedding to $\mathbb{G}$, we obtain a cograph in $\Graph$:
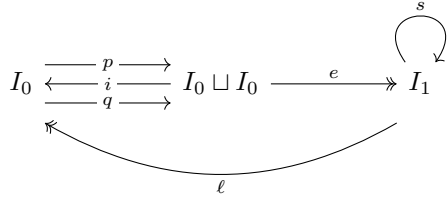
\[\begin{tikzcd}
	{I_0} && {I_0 \sqcup I_0} && {I_1}
	\arrow["p"{description}, shift left=3, from=1-1, to=1-3]
	\arrow["q"{description}, shift right=3, from=1-1, to=1-3]
	\arrow["i"{description}, from=1-3, to=1-1]
	\arrow["e", two heads, from=1-3, to=1-5]
	\arrow["\ell", shift left=4, bend left, two heads, from=1-5, to=1-1]
    \arrow["s", from=1-5, to=1-5, loop, in=55, out=125, distance=10mm]
\end{tikzcd}\]
Cographs are preserved by $G \otimes -$ and $- \otimes G$, so we obtain a double cograph as in \cref{diagram}. As before $I_1 \otimes I_1$ has exactly 4 vertices and $C_4$ is a subgraph of $I_1 \otimes I_1$; however, now there are also two endomorphisms $s_h, s_v$ of $I_1 \otimes I_1$ that act on vertices as follows:
    \begin{align*}
        s_h : \begin{cases}
        a \mapsto b \mapsto a \\
        c \mapsto d \mapsto c
        \end{cases}
         & & s_v : \begin{cases}
         a \mapsto c \mapsto a \\
         b \mapsto d \mapsto b
         \end{cases}
    \end{align*}
Then, the only thing left to determine is if there are any diagonal edges. If there are no diagonal edges, then the monoidal structure is the box product. If there was one diagonal edge, say $a \sim d$, then by applying the endomorphism $s_h$ we also have $b \sim c$. Therefore, the only possibility is if there are both diagonal edges. If there are both diagonal edges, then the monoidal structure is the categorical product.
Thus the box product and the categorical product are the only biclosed monoidal structures on the category of undirected reflexive graphs.
\end{proof}

\section{Graphs with optional loops}

One can ask whether these arguments carry over to the case of graphs that are not necessarily reflexive.
Unfortunately, to our knowledge, new arguments are required to address the problem of the classification of biclosed monoidal structures on that category.
Before reviewing the relevant issues, let us formulate our conjecture:

\begin{conjecture}
 The category of (directed or undirected) graphs with optional loops carries exactly three biclosed monoidal structures: the categorical product, the box product, and the strong product.
\end{conjecture}

Many of the arguments used above apply to this case, though often with additional complications and limitations:

\textbf{Units.}
Any biclosed monoidal structure on the category of graphs with optional loops necessarily has either a looped or an unlooped vertex as its unit.
The proof is similar to our \cref{lem:unit}, and follows from \cite[Prop.~1.7]{kapulkin-kershaw:monoidal-graphs}.
In fact, each of these is a unit of at least one monoidal structure mentioned in our conjecture.

\textbf{Reduction to representables.}
The monoidal structure is completely determined by the products of the graph with a single unlooped vertex and the graph with two vertices and a single edge between them.
This is again a consequence of the density theorem (cf.~\cref{lem:determined-by-square}), but since the graph with a single unlooped vertex is not guaranteed to be the monoidal unit, we need to consider all four cases.

\textbf{Special cases.}
In the case when the unlooped vertex is the monoidal unit, one can actually show that the product of the graph consisting of a single edge with itself has four vertices and (at least) four edges.

The key difficulty lies in the fact that there is no analogue of the morphism $\ell$ in our new indexing category, which plays a key role in the final step of the proof that the product of an edge with itself has 4 vertices (\cref{lem:J1-J1-four-vertices}).
Consequently, new tools are required to address the above conjecture.

 \bibliographystyle{amsalphaurlmod}
 \bibliography{all-refs.bib}

\end{document}